\documentclass[a4paper]{article}
\usepackage{geometry}
\usepackage[tbtags]{amsmath}
\usepackage{amsthm}
\usepackage{amssymb}
\usepackage{cite}
\usepackage{booktabs}
\usepackage[pdftex,colorlinks]{hyperref}
\usepackage{subfigure}
\usepackage{graphicx}
\usepackage{epstopdf}
\usepackage{epsf}
\usepackage{multicol}
\usepackage{multirow}
\numberwithin{equation}{section}

\newcommand{\fanshu}[2][0]{\Vert#2\Vert_{#1}}

\newcommand{\mesh}[1][h]{\mathcal{T}_{#1}}

\newcommand{\pror}{\mathcal{P}_h}

\newtheorem{theorem}{\noindent{\bf Theorem}}[section]
\newtheorem{definition}{\noindent{Definition}}[section]
\newtheorem{lemma}{\noindent{Lemma}}[section]

\theoremstyle{remark}
\newtheorem{remark}{\noindent{Remark}}[section]

\begin{document}\large 
\title{{\large\textbf{ Solving Biharmonic Eigenvalue Problem With Navier Boundary Condition Via Poisson Solvers On Non-Convex Domains}}}
\author{\small{Baiju Zhang}\thanks{Beijing Computational Science Research Center, Beijing 100193, China, \tt baijuzhang@csrc.ac.cn} \and \small{Hengguang Li}\thanks{Department of Mathematics, Wayne State University, Detroit, MI 48202, USA, \tt li@wayne.edu} \and
\small{Zhimin Zhang}\thanks{Corresponding author. Beijing Computational Science Research Center, Beijing 100193, China, {\tt zmzhang@csrc.ac.cn}; Department of Mathematics, Wayne State University, Detroit, MI 48202, USA, \tt zzhang@math.wayne.edu}}
\date{}\maketitle 

\begin{abstract}\large It is well known that the usual mixed method for solving the biharmonic eigenvalue problem by decomposing the operator into two Laplacians may generate spurious eigenvalues on non-convex domains.
To overcome this difficulty, we adopt a recently developed mixed method, which decomposes the biharmonic equation into three Poisson equations and still recovers the original solution.
Using this idea, we design an efficient biharmonic eigenvalue algorithm, which contains only Poisson solvers. With this approach, eigenfunctions can be confined in the correct space and thereby spurious modes in non-convex domains are avoided.
\textit{A priori} error estimates for both eigenvalues and eigenfunctions on quasi-uniform meshes are obtained; in particular, a convergence rate of $\mathcal{O}({h}^{2\alpha})$ ($ 0<\alpha<\pi/\omega$, $\omega > \pi$ is the angle of the reentrant corner) is proved for the linear finite element.
Surprisingly, numerical evidence demonstrates a $\mathcal{O}({h}^{2})$ convergent rate for the quasi-uniform mesh with the regular refinement strategy even on non-convex polygonal domains.
\end{abstract}

\section{Introduction}\label{introduction}
The biharmonic eigenvalue problem is one of the fundamental model problems in thin plate theories of elasticity. As it is a fourth-order problem, a direct conforming discretization needs a finite element space with $C^1$ continuity. However, due to its high computational cost and complexity of the finite element space structure, this method is not very practical, especially in the three-dimensional case. To overcome this difficulty, various finite element methods have been introduced and analyzed, such as nonconforming methods \cite{morley}, discontinuous Galerkin methods \cite{mozolevski2007hp,georgoulis2009DG}, $C^0$ interior penalty Galerkin ($C^0$IPG) methods \cite{brenner2005c0interior,brenner2015c0interior}, and mixed finite element methods \cite{ciarlet1974mixed}.

In this paper, we are interested in a mixed finite element approximation of the biharmonic eigenvalue problem with Navier boundary conditions in a non-convex polygonal domain. The usual mixed element method \cite{ciarlet1974mixed} is particularly appealing for the biharmonic equations with Navier boundary conditions, because such boundary conditions allow one to decompose a fourth-order equation into two second-order equations that are completely decoupled. This means that a reasonable numerical solution can be obtained by merely applying the finite element Poisson solver in the mixed formulation. Unfortunately, it has been reported in \cite{gerasimov2012corners,nazarov2007ahinged,zhang2008invalidity} that the numerical behavior of this mixed method can be affected by the domain geometry. In a convex domain, it was proved that the corresponding numerical solutions converge to the solution of the primal formulation. However, this is invalid when the domain is non-convex. In fact, the usual mixed formulation defines a weak solution in a space larger than that for the primal formulation. If the domain is non-convex, this mismatch in the function spaces leads to the emergence of singular functions, which in turn causes the solution to be different from that in the primal formulation. As for biharmonic eigenvalue problems, some references \cite{brenner2015c0interior,yang2018theadaptive} also observed similar results. It has been pointed out in these references that the usual mixed method may generate spurious eigenvalues when the domain is non-convex. It seems that the usual mixed method is not suitable for the biharmonic problems with Navier boundary conditions in non-convex domains.

Very recently, H. Li et al. \cite{li2020biharmonic} proposed a modified mixed formulation for the standard biharmonic problem that is decoupled as usual, and at the same time ensures that the associated solution is equal to the solution of the original biharmonic equation in both convex and non-convex domains. The main idea is to introduce an additional intermediate Poisson problem that confines the solution to the correct space. The whole process involves only the Poisson solver. One may ask whether it is possible to extend this method to the biharmonic eigenvalue problem. This paper attempts to answer this question. We show that, both theoretically and numerically, the idea can indeed be used in eigenvalue problems to obtain reliable results without spurious modes. Moreover, in our numerical experiments, we find that even if uniform meshes are used in non-convex domains, the eigenvalues obtained by the modified mixed method can converge at the optimal rate $\mathcal{O}(h^2)$. This is very different from other existing methods, when at most $\mathcal{O}({h}^{2\beta})$ ($ 0<\beta<1$) convergence can be achieved if a graded mesh is not used \cite{brenner2015c0interior}. We would like to emphasize that the new method uses only the Poisson solver on uniform or quasi-uniform meshes, for which there are many packages available with low computational cost ($\mathcal{O}(N)$, where $N$ is the total number of degrees of freedom). This is a huge advantage over existing methods for the biharmonic eigenvalue problem.

The rest of the paper is arranged as follows. In Section \ref{preliminary}, we review the model problem with related notations, introduce the modified mixed formulation proposed in \cite{li2020biharmonic}, and extend it to the biharmonic eigenvalue problem. In Section \ref{sectionApriori}, we derive \textit{a priori} error estimates of numerical eigenvalues and eigenfunctions. In Section \ref{SectionNumericalExp}, several numerical experiments are presented to confirm our theoretical analysis.

For simplicity of notation, we shall use $\lesssim$ to denote less than or equal to up to a constant independent of the mesh size, variables, or other parameters appearing in the inequality.
\section{Preliminary}\label{preliminary}
Let $\Omega\subset\mathbb{R}^2$ be a polygonal domain with boundary $\partial\Omega$. Denote by $H^r(\Omega)$ the usual Sobolev space of real order $r$ with norm $\fanshu[r]{\cdot}$. Conventionally, we set $H^0(\Omega)=L^2(\Omega)$ and $H^1_0(\Omega)=\{v\in H^1(\Omega):\ v|_{\partial\Omega}=0\}$.

Consider the following biharmonic equation and associated eigenvalue problem:
\begin{equation}
\Delta^2\phi=f\text{ in }\Omega,\qquad \phi=\Delta\phi=0\text{ on }\partial\Omega;\label{biharmonicEqu}
\end{equation}
\begin{equation}
\Delta^2u=\lambda u\text{ in }\Omega,\qquad u=\Delta u=0\text{ on }\partial\Omega,\label{biharmonicEig}
\end{equation}
where $f\in L^2(\Omega)$.
The weak formulation of the biharmonic problem \eqref{biharmonicEqu} is to seek $\phi\in H_0^1(\Omega)\cap H^2(\Omega)$ such that
\begin{equation}
(\Delta \phi,\Delta \psi)=(f,\psi)\quad \psi\in H_0^1(\Omega)\cap H^2(\Omega),\label{biharmonicEquPrimal}
\end{equation}
and the weak formulation of \eqref{biharmonicEig} is to seek $u\in H_0^1(\Omega)\cap H^2(\Omega)$ such that
\begin{equation}
(\Delta u,\Delta v)=\lambda(u,v)\quad v\in H_0^1(\Omega)\cap H^2(\Omega).\label{biharmonicEigPrimal}
\end{equation}
The above formulations require the use of $H^2$-conforming methods, which are quite complicated. In practice, this is far from desirable. Instead of using such methods, a popular procedure is to use mixed formulation.

Usually, set $\bar{\rho}=-\Delta\bar{\phi}$, then equation \eqref{biharmonicEqu} can be written as the following mixed formulation: find $(\bar{\rho},\bar{\phi})\in H_0^1(\Omega)\times H_0^1(\Omega)$ such that
\begin{subequations}\label{usualform}
\begin{align}
(\nabla \bar{\phi},\nabla\psi)=(\bar{\rho},\psi),\qquad\forall\psi\in H_0^1(\Omega),\\
(\nabla \bar{\rho},\nabla\nu)=(f,\nu),\qquad\forall\nu\in H_0^1(\Omega).
\end{align}
\end{subequations}
Similarly, by introducing the auxiliary variable $\bar{\sigma}=-\Delta \bar{u}$, we can derive the following eigenvalue problem: find $(\bar{\lambda},\bar{\sigma},\bar{u})\in\mathbb{R}\times H_0^1(\Omega)\times H_0^1(\Omega)$ such that
\begin{subequations}\label{usualformEigen}
\begin{align}
(\nabla \bar{u},\nabla\psi)=(\bar{\sigma},\psi),\qquad\forall\psi\in H_0^1(\Omega),\\
(\nabla \bar{\sigma},\nabla\nu)=(\bar{\lambda}\bar{u},\nu),\qquad\forall\nu\in H_0^1(\Omega).
\end{align}
\end{subequations}
We see that both \eqref{usualform} and \eqref{usualformEigen} consist of two completely decoupled Poisson equations, and are much simpler to solve than \eqref{biharmonicEqu} and \eqref{biharmonicEig}, respectively. However, according to \cite{nazarov2007ahinged,zhang2008invalidity,gerasimov2012corners}, the solution of \eqref{biharmonicEqu} is not always equal to the solution of \eqref{usualform}. The equivalence may depend on the domain geometry. When $\Omega$ is a convex domain, the solution of \eqref{biharmonicEqu} coincides with the solution of \eqref{usualform}, and then numerical approximations of \eqref{usualform} converge to the solution of \eqref{biharmonicEqu}. On the other hand, if $\Omega$ is a non-convex domain, the solution of \eqref{usualform} may be a spurious solution of \eqref{biharmonicEqu}. As a consequence, the mixed finite element method based on \eqref{usualform} may generate numerical solutions that do not converge to the solution of \eqref{biharmonicEqu}. When it comes to eigenvalue problems, similar results have been reported in some references \cite{brenner2015c0interior,yang2018theadaptive}. The mixed method based on \eqref{usualformEigen} may generate spurious eigenvalues if $\Omega$ is non-convex. The main reason for the occurrence of the spurious eigenvalues is that the usual mixed formulation \eqref{usualformEigen} defines eigenfunctions in a larger space than that for \eqref{biharmonicEig}. To avoid spurious eigenvalues, we exploit the method of\cite{li2020biharmonic} which introduces an additional intermediate Poisson problem to capture the singular term so that the solution of \eqref{usualform} can be restricted in the correct space. Using this approach, the eigenfunctions of \eqref{usualformEigen} are guaranteed to be in the correct space, which results in correct corresponding eigenvalues.

To begin with, we assume that $\Omega$ has a re-entrant corner $Q$ and the corresponding interior angle $\omega\in(\pi,2\pi)$. Without loss of generality, we set $Q$ as the origin. Let $(r,\theta)$ be the polar coordinates with $Q$ as the center, and $\omega$ be composed of two half lines $\theta=0$, and $\theta=\omega$. Given $R>0$, we define a sector $K_{\omega}^R\subset\Omega$ with radius $R$ as
\[K_{\omega}^R=\{(r\cos\theta,r\sin\theta)|\ 0\le r\le R,\ 0\le\theta\le\omega\}.\]

Following \cite{li2020biharmonic}, we introduce an $L^2$ function which plays an important role in constructing the equivalent mixed formulation in $\Omega$.
\begin{definition}\label{definitionxi}
Given the parameters $\tau\in(0,1)$ and $R$ such that $K_{\omega}^R\subset\Omega$, we define an $L^2$ function in $\Omega$,
\begin{equation}
\xi(r,\theta;\tau,R):=s^{-}(r,\theta;\tau,R)+\zeta(r,\theta;\tau,R),
\end{equation}
where
\begin{equation}
s^{-}(r,\theta;\tau,R)=\chi(r;\tau,R)r^{-\frac{\pi}{\omega}}\sin(\frac{\pi}{\omega}\theta)\in L^2(\Omega),
\end{equation}
with $\chi(r;\tau,R)\in C^{\infty}(\Omega)$ satisfying $\chi(r;\tau,R)=1$ for $0\le r\le\tau R$ and $\chi(r;\tau,R)=0$ for $r>R$, and $\zeta\in H_0^1(\Omega)$ satisfies
\begin{equation}
-\Delta \zeta=\Delta s^{-} \text{ in }\Omega, \qquad\zeta=0\text{ on }\partial\Omega.\label{zetaEquation}
\end{equation}
\end{definition}
It is east to see that $s^{-}\in C^{\infty}(\Omega\setminus K_{\omega}^{\delta})$ for any $\delta>0$ and $s^{-}=0$ for $(r,\theta)\in \Omega\setminus K_{\omega}^{R}$. Furthermore $\Delta s^{-}=0$ if $r< \tau R$ or $r>R$.

Denote
\[a(w,\psi)=(w,\psi)- (c(w)\xi,\psi),\ b(\psi,v)=(\nabla\psi,\nabla v),\]
where
\begin{equation}
c(w)=\frac{(w,\xi)}{\fanshu{\xi}^2}.\label{definitionCw}
\end{equation}
Based on the above definition, the authors in \cite{li2020biharmonic} introduced the following modified mixed method for \eqref{biharmonicEqu},
\begin{align}\label{biharmonicEquMixed}
\begin{cases}
-\Delta \rho=f\text{ in }\Omega,\\
\quad\ \ \rho=0\text{ on }\partial\Omega;
\end{cases}
\text{   and   }
\begin{cases}
-\Delta \phi=\rho-c(\rho)\xi\text{ in }\Omega,\\
\quad\ \ \phi=0\text{ on }\partial\Omega,
\end{cases}
\end{align}
and corresponding variational formulation: find $(\rho,\phi)\in H_0^1(\Omega)\times H_0^1(\Omega)$ such that
\begin{subequations}\label{modifiedMixedweak}
\begin{align}
-a(\rho,\psi)+b(\psi,\phi)&=0\qquad\forall\psi\in H_0^1(\Omega),\label{modifiedMixedweakA}\\
b(\rho,v)&=(f,v)\qquad\forall v\in H_0^1(\Omega). \label{modifiedMixedweakB}
\end{align}
\end{subequations}

Motivated by this, we propose the following modified eigenvalue problem for \eqref{biharmonicEig},
\begin{align}\label{modifiedEigenEqu}
\begin{cases}
-\Delta \sigma=\lambda u\text{ in }\Omega,\\
\quad\ \ \sigma=0\text{ on }\partial\Omega;
\end{cases}
\text{   and   }
\begin{cases}
-\Delta u=\sigma-c(\sigma)\xi\text{ in }\Omega,\\
\quad\ \ u=0\text{ on }\partial\Omega,
\end{cases}
\end{align}
and the corresponding modified mixed weak formulation: find $(\lambda,\sigma,u)\in \mathbb{R}\times H_0^1(\Omega)\times H_0^1(\Omega)$ such that
\begin{subequations}\label{modifiedEigen}
\begin{align}
-a(\sigma,\psi)+b(\psi,u)&=0\qquad\forall\psi\in H_0^1(\Omega),\\
b(\sigma,v)&=(\lambda u,v)\qquad\forall v\in H_0^1(\Omega).
\end{align}
\end{subequations}
\begin{remark}
If $\Omega$ has $N_c$ re-entrant corners $(N_c>1)$, we only need to rewrite the second equation of \eqref{modifiedEigenEqu} as follows
\begin{align}
\begin{cases}
-\Delta u=\sigma-\sum_{i=1}^{N_c}c_i(\sigma)\xi_i\text{ in }\Omega,\\
\quad\ \ u=0\text{ on }\partial\Omega,
\end{cases}
\end{align}
where for $i=1,2,\cdots,N_c$, $\xi_i$ can be defined as in Definition \ref{definitionxi}, and $c_i(\sigma)$ can be solved by the following linear equations
\begin{equation}
\sum_{j=1}^{N_c}(\xi_i,\xi_j)c_j(\sigma)=(\sigma,\xi_i).
\end{equation}
\end{remark}
Given $f\in L^2(\Omega)$, we can rewrite the following component solution operators for \eqref{modifiedMixedweak}:
\begin{align}
T: L^2(\Omega)\rightarrow H_0^1(\Omega), T f=\phi; \qquad S: L^2(\Omega)\rightarrow H_0^1(\Omega), S f=\rho.\label{solutionOperators}
\end{align}
Using the above operators, we can rewrite \eqref{modifiedEigen} as the following equivalent operator forms
\begin{align}
Tu=\lambda^{-1}u,\qquad\sigma=S(\lambda u).\label{operatorFormNew}
\end{align}
\begin{remark}
According to Theorem 2.7 in \cite{li2020biharmonic}, \eqref{modifiedMixedweak} is equivalent to \eqref{biharmonicEquPrimal}. From this, one can easily see that the solution operator $T$ defined in \eqref{solutionOperators} is also the solution operator of \eqref{biharmonicEquPrimal}, and vice versa. Therefore they induce the same eigenvalue problem. However, this does not mean that any finite element method that can approximate \eqref{modifiedMixedweak} is suitable for the eigenvalue problem \eqref{modifiedEigen}. As pointed out in \cite{boffi2010finite}, there is an intrinsic difference between the source problem and eigenvalue problem, and hence, it is worth analyzing the convergence of the finite element method for \eqref{modifiedEigen}.
\end{remark}

We now introduce the finite element method for \eqref{modifiedMixedweak} and \eqref{modifiedEigen}. Let $\{\mesh\}$ be a series of shape-regular meshes of $\Omega$: there exists a constant $\gamma^*$ such that
\[\frac{h_K}{d_K}\le\gamma^*\ \forall K\in\cup_h\mesh,\]
where, for each $K\in\mesh$, $h_K$ is the diameter of $K$ and $d_K$ is the diameter of the biggest ball contained in $K$. As usual, we set $h=\max_{K\in\mesh}h_K$. Let $S_0^h\subset H_0^1(\Omega)$ be the $C^0$ Lagrange finite element space associated with $\mesh$,
\begin{equation}
S_0^h:=\{v\in C^0(\bar{\Omega})\cap H_0^1(\Omega):\ v|_K\in P_1(K), \forall K\in\mesh\},
\end{equation}
where $P_1(K)$ is the space of polynomials of degree no more than $1$ on $K$.

To define the modified mixed finite element methods of \eqref{modifiedMixedweak} and \eqref{modifiedEigen}, we start with computing the finite element solution $\zeta_h\in S_0^h$ of the Poisson equation
\begin{equation}
b(\zeta_h,v_h)=(\Delta s^-,v_h)\qquad \forall v_h\in S_0^h,\label{equZeta}
\end{equation}
and set $\xi_h=\zeta_h+s^-.$
Denote
\[a_h(w,\psi)=(w,\psi)- (c_h(w)\xi_h,\psi),\]
where
\begin{equation}
c_h(w)=\frac{(w,\xi_h)}{\fanshu{\xi_h}^2}.\label{definitionChwh}
\end{equation}
Then the modified mixed finite element method of \eqref{modifiedMixedweak} is defined as follows:
find $(\rho_h,\phi_h)\in S_0^h\times S_0^h$ such that
\begin{subequations}\label{modifiedMixedFem}
\begin{align}
-a_h(\rho_h,\psi_h)+b(\psi_h,\phi_h)&=0\qquad\forall\psi_h\in S_0^h,\label{modifiedMixedFemA}\\
b(\rho_h,v_h)&=(f,v_h)\qquad\forall v_h\in S_0^h.\label{modifiedMixedFemB}
\end{align}
\end{subequations}
Inspired by the above, we propose the following modified mixed method for \eqref{modifiedEigen}:
find $(\lambda_h,\sigma_h,u_h)\in \mathbb{R}\times S_0^h\times S_0^h$ such that
\begin{subequations}\label{modifiedMixedFemEigen}
\begin{align}
-a_h(\sigma_h,\psi_h)+b(\psi_h,u_h)&=0\qquad\forall\psi_h\in S_0^h,\\
b(\sigma_h,v_h)&=(\lambda_h u_h,v_h)\qquad\forall v_h\in S_0^h.
\end{align}
\end{subequations}
Clearly, given $f\in L^2(\Omega)$, \eqref{modifiedMixedFem} is uniquely solvable. Therefore we can define the following component solution operators:
\begin{align*}
T_h: L^2(\Omega)\rightarrow S_0^h(\Omega), T_hf=\phi_h;&\qquad S_h: L^2(\Omega)\rightarrow S_0^h(\Omega), S_h f=\rho_h.
\end{align*}
Using these operators, we can rewrite \eqref{modifiedMixedFemEigen} as the following equivalent operator forms
\begin{align}
T_hu_h=\lambda_h^{-1}u_h,\qquad \sigma_h=S_h(\lambda_h u_h).
\end{align}
It can be easily verified that both $T: L^2(\Omega)\rightarrow L^2(\Omega)$ in \eqref{operatorFormNew} and $T_h: L^2(\Omega)\rightarrow L^2(\Omega)$ are compact self-adjoint operators.

Given $g\in L^2(\Omega)$, consider the following Poisson equation: find $w\in H_0^1(\Omega)$ such that
\begin{equation}
b(w,v)=(g,v),\qquad\forall v\in H_0^1(\Omega),\label{PoissonEqu}
\end{equation}
and its finite element approximation: find $w_h\in S_0^h$ such that
\begin{equation}
b(w_h,v_h)=(g,v_h),\qquad\forall v_h \in S_0^h\label{PoissonEquFem}.
\end{equation}

Define the Ritz projection operator $\pror: H_0^1(\Omega)\rightarrow S_0^h$ such that
\begin{equation}
b(w-\pror w,v_h)=0,\qquad\forall v_h \in S_0^h\label{RitzPro}.
\end{equation}
It is easy to see that the solution operators corresponding to \eqref{PoissonEqu} and \eqref{PoissonEquFem} are just $S$ and $S_h$, respectively. In addition, for any $g\in L^2(\Omega)$ there hold
\begin{align}
&w= Sg\qquad w_h=S_hg,\label{SShoperator}\\
&\fanshu[1]{Sg}\le C\fanshu[0]{g},\qquad \fanshu[1]{S_hg}\le C\fanshu[0]{g},\label{Sbound}
\end{align}
where $C$ is a positive constant independent of $g$ and mesh size $h$.
It follows from \eqref{modifiedMixedweakB}, \eqref{modifiedMixedFemB} and \eqref{RitzPro} that
\begin{equation}
\rho_h=\pror \rho,\qquad S_h=\pror S.\label{ProAndS}
\end{equation}
According to \cite{grisvard1985elliptic}, we have the following elliptic regularity estimate:
\begin{equation}
\fanshu[1+\alpha]{Sg}\le C_{\Omega,\alpha}\fanshu[0]{g},\qquad\forall g\in L^2(\Omega),\label{ellipticRegularity}
\end{equation}
where $\alpha\in(0,\pi/\omega)$.
From \eqref{ellipticRegularity} and interpolation theory we have
\begin{equation}
\eta_0(h)\equiv\sup_{f\in L^2(\Omega),\fanshu[0]{f}=1}\inf_{\psi\in S_0^h}{\fanshu[1]{Sf-\psi}}\lesssim h^{\alpha}.\label{eta0Estimate}
\end{equation}

\section{A priori error analysis}\label{sectionApriori}
The error estimates of the modified mixed method of biharmonic equation \eqref{biharmonicEqu} can be expressed as follows:
\begin{theorem}
Let $(\rho,\phi)$ and $(\rho_h,\phi_h)$ be the solutions of \eqref{modifiedMixedweak} and \eqref{modifiedMixedFem}, respectively, then
\begin{align}
\fanshu[1]{\rho-\rho_h}&\lesssim\fanshu{\nabla(\rho-\pror\rho)},\\
\fanshu{\rho-\rho_h}&\lesssim\eta_0(h)\fanshu{\nabla(\rho-\rho_h)},\label{L2RhoErr}\\
\fanshu[1]{\phi-\phi_h}&\lesssim\fanshu[1]{\phi-\pror\phi}+\fanshu{\rho-\rho_h}+\fanshu{\xi-\xi_h},\label{H1PhiErr}\\
\fanshu{\phi-\phi_h}&\lesssim\eta_0(h)\fanshu[1]{\phi-\pror\phi}+\fanshu{\rho-\rho_h}+\fanshu{\xi-\xi_h}.\label{L2PhiErr}
\end{align}
\end{theorem}
\begin{proof}
The first two estimates are the direct result of Theorem 2.1 of \cite{yang2018theadaptive}. To prove the rest of the theorem, we rewrite \eqref{modifiedMixedweakA} and \eqref{modifiedMixedFemA} as follows, respectively:
\begin{align}
b(\phi,\psi)=(\rho-c(\rho)\xi,\psi),\qquad\forall\psi\in H_0^1(\Omega),\\
b(\phi_h,\psi_h)=(\rho_h-c_h(\rho_h)\xi_h,\psi_h),\qquad\forall\psi_h\in S_0^h.
\end{align}
Clearly, by \eqref{SShoperator}, we have
\begin{equation}
\phi=S(\rho-c(\rho)\xi),\qquad \phi_h=S_h(\rho_h-c_h(\rho_h)\xi_h).\label{phiandS}
\end{equation}
From \eqref{Sbound} and \eqref{ProAndS}, for $s=0,1$, we have
\begin{align*}
\fanshu[s]{\phi-\phi_h}&=\fanshu[s]{S(\rho-c(\rho)\xi)-S_h(\rho_h-c_h(\rho_h)\xi_h)}\\
&\le\fanshu[s]{S(\rho-c(\rho)\xi)-S_h(\rho-c(\rho)\xi)}+\fanshu[s]{S_h(\rho-c(\rho)\xi)-S_h(\rho_h-c_h(\rho_h)\xi_h)}\\
&\le\fanshu[s]{S(\rho-c(\rho)\xi)-\pror S(\rho-c(\rho)\xi)}+\fanshu{\rho-\rho_h}+\fanshu{c(\rho)\xi-c_h(\rho_h)\xi_h}\\
&\le\fanshu[s]{\phi-\pror\phi}+\fanshu{\rho-\rho_h}+\fanshu{c(\rho)(\xi-\xi_h)}+|c(\rho)-c_h(\rho_h)|\fanshu{\xi_h}.
\end{align*}
Using (3.35) in \cite{li2020biharmonic} we get
\begin{equation}
|c(\rho)-c_h(\rho_h)|\lesssim\fanshu{\rho-\rho_h}+\fanshu{\xi-\xi_h}.\label{crhominuschrhoh}
\end{equation}
Thus we obtain
\begin{align}
\fanshu[s]{\phi-\phi_h}\lesssim\fanshu[s]{\phi-\pror\phi}+\fanshu{\rho-\rho_h}+\fanshu{\xi-\xi_h}.\label{IntermediateResult}
\end{align}
Therefore, taking $s=1$ we get \eqref{H1PhiErr}, taking $s=0$ and using the Nitsche technique we arrive at \eqref{L2PhiErr}.
\end{proof}

According to the spectral approximation theory \cite{babuska1991eigenvalue,chatelin2011spectral}, we derive the following theorem.
\begin{theorem}\label{TheoremErrAll}
Let $(\lambda_h,\sigma_h,u_h)$ be the j-th eigenpair of \eqref{modifiedMixedFemEigen} with $\fanshu{u_h}=1$, $\lambda$ be the j-th eigenvalue of \eqref{modifiedEigen}. Then when $h$ is sufficiently small, there exists an eigenfunction $(\sigma,u)$ corresponding to $\lambda$ such that
\begin{align}
|\lambda-\lambda_h|&\lesssim \eta_0(h)(\fanshu[1]{u-\pror u}+\fanshu[1]{\sigma-\pror\sigma})+\fanshu{\xi-\xi_h},\label{lambdaErr}\\
\fanshu{u-u_h}&\lesssim \eta_0(h)(\fanshu[1]{u-\pror u}+\fanshu[1]{\sigma-\pror\sigma})+\fanshu{\xi-\xi_h},\label{L2uErr}\\
\fanshu[1]{u-u_h}&\lesssim \fanshu[1]{u-\pror u}+\fanshu{\sigma-\pror\sigma}+\fanshu{\xi-\xi_h},\label{H1uErr}\\
\fanshu{\sigma-\sigma_h}&\lesssim\eta_0(h)(\fanshu[1]{u-\pror u}+\fanshu[1]{\sigma-\pror\sigma})+\fanshu{\xi-\xi_h},\label{L2sigErr}\\
\fanshu[1]{\sigma-\sigma_h}&\lesssim\eta_0(h)\fanshu[1]{u-\pror u}+\fanshu[1]{\sigma-\pror\sigma}+\fanshu{\xi-\xi_h}.\label{H1sigErr}
\end{align}
\end{theorem}
\begin{proof}
Recall that $Tf=\phi=S(\rho-c(\rho)\xi),\ T_hf=\phi_h=S_h(\rho_h-c_h(\rho_h)\xi_h)$ and $\rho=Sf$, then, from \eqref{Sbound} we have for $s=0,1$
\begin{align*}
\fanshu[s]{Tf\!-\!T_hf}&\lesssim\fanshu[s]{S(\rho\!-\!c(\rho)\xi)\!-\!\pror S(\rho-c(\rho)\xi)}+\fanshu[s]{S_h(\rho-c(\rho)\xi)-S_h(\rho_h-c_h(\rho_h)\xi_h)}\\
&\lesssim \fanshu[s]{S(\rho\!-\!c(\rho)\xi)\!-\!\pror S(\rho-c(\rho)\xi)}+\fanshu{\rho-\rho_h}+\fanshu{c(\rho)\xi-c_h(\rho_h)\xi_h}.
\end{align*}
According to Theorem 2.7 in \cite{li2020biharmonic}, we have $\phi=S(\rho-c(\rho)\xi)\in H_0^1(\Omega)\cap H^2(\Omega)$. Then, using the interpolation theory, \eqref{Sbound}, \eqref{ellipticRegularity} and \eqref{definitionCw} we get
\begin{align*}
&\fanshu[1]{S(\rho-c(\rho)\xi)-\pror S(\rho-c(\rho)\xi)}\lesssim h\fanshu[2]{S(\rho-c(\rho)\xi)}\\
&\lesssim h\fanshu{\rho-c(\rho)\xi}\lesssim h\fanshu{\rho}=h\fanshu{S f}\\
&\lesssim h\fanshu{f},
\end{align*}
and
\begin{align*}
&\fanshu{\rho-\rho_h}\lesssim h^{2\alpha} \fanshu[1+\alpha]{\rho}=h^{2\alpha} \fanshu[1+\alpha]{Sf}\lesssim h^{2\alpha} \fanshu{f}.
\end{align*}
By the Nitche technique, we obtain
\[\fanshu[0]{S(\rho-c(\rho)\xi)-\pror S(\rho-c(\rho)\xi)}\lesssim h^{1+\alpha}\fanshu{f}.\]
Using Lemma 3.3 and (3.35) in \cite{li2020biharmonic} we obtain
\begin{align*}
|c(\rho)-c_h(\rho_h)|&\le \frac{1}{\fanshu{\xi}}\fanshu{\rho-\rho_h}+\frac{\fanshu{\rho_h}}{\fanshu{\xi_h}^2}\fanshu{\xi-\xi_h}+\frac{(\fanshu{\xi_h}+\fanshu{\xi})\fanshu{\rho_h}}{\fanshu{\xi}\fanshu{\xi_h}^2}\fanshu{\xi-\xi_h}\\
&\lesssim h^{2\alpha}(\fanshu[1+\alpha]{\rho}+\fanshu{\rho})\lesssim h^{2\alpha}\fanshu{f}.
\end{align*}
Applying the above inequality leads to
\begin{align*}
\fanshu{c(\rho)\xi-c_h(\rho_h)\xi_h}\le\fanshu{c(\rho)\xi-c(\rho)\xi_h}+\fanshu{c(\rho)\xi_h-c_h(\rho_h)\xi_h}\\
\lesssim \fanshu{\xi-\xi_h}|c(\rho)|+|c(\rho)-c_h(\rho_h)|
\lesssim h^{2\alpha}\fanshu{f}.
\end{align*}
Combining all the above inequalities we have
\begin{align}
\fanshu[1]{Tf-T_hf}&\lesssim h\fanshu{f},\\
\fanshu[0]{Tf-T_hf}&\lesssim h^{2\alpha}\fanshu{f}.
\end{align}
From the spectral approximation theory \cite{babuska1991eigenvalue,chatelin2011spectral} we get
\begin{equation}
|\lambda-\lambda_h|\lesssim \lambda\fanshu{(T-T_h)u}\label{llhErr}
\end{equation}
and
\begin{equation}
\fanshu[s]{u-u_h}\lesssim \lambda\fanshu{(T-T_h)u}\qquad s=0,1.\label{uuhErr}
\end{equation}
Take $f=\lambda u$ in \eqref{modifiedMixedweak}, then using \eqref{L2PhiErr} and \eqref{L2RhoErr} we obtain
\begin{align}
\lambda\fanshu{(T-T_h)u}&\lesssim\eta_0(h)\fanshu[1]{T(\lambda u)-\pror T(\lambda u)}+\fanshu{S(\lambda u)-S_h(\lambda u)}+\fanshu{\xi-\xi_h}\notag\\
&\lesssim\eta_0(h)\fanshu[1]{u-\pror u}+\fanshu{\sigma-\pror\sigma}+\fanshu{\xi-\xi_h}\notag\\
&\lesssim\eta_0(h)(\fanshu[1]{u-\pror u}+\fanshu[1]{\sigma-\pror\sigma})+\fanshu{\xi-\xi_h}.\label{lTThErr}
\end{align}
Inserting \eqref{lTThErr} into \eqref{llhErr} and \eqref{uuhErr} we have \eqref{lambdaErr} and \eqref{L2uErr}, respectively.
By \eqref{H1PhiErr} we deduce
\begin{align}
\lambda\fanshu[1]{(T-T_h)u}&\lesssim\fanshu[1]{T(\lambda u)-\pror T(\lambda u)}+\fanshu{S(\lambda u)-S_h(\lambda u)}+\fanshu{\xi-\xi_h}\notag\\
&=\fanshu[1]{u-\pror u}+\fanshu{\sigma-\pror\sigma}+\fanshu{\xi-\xi_h},
\end{align}
from which and \eqref{uuhErr} we can get \eqref{H1uErr}.
Applying \eqref{Sbound}, the triangle inequality, \eqref{lambdaErr} and \eqref{L2uErr} yields
\begin{align}
\fanshu[s]{\sigma_h-\pror\sigma}&=\fanshu[s]{S_h(\lambda_hu_h)-\pror S(\lambda u)}\notag\\
&\lesssim \fanshu{\lambda_h u_h-\lambda u}\le |\lambda|\fanshu{u-u_h}+|\lambda-\lambda_h|\fanshu{u_h}\notag\\
&\lesssim \eta_0(h)(\fanshu{u-\pror u}+\fanshu[1]{\sigma-\pror\sigma})+\fanshu{\xi-\xi_h},\qquad s=0,1.\label{inequality1}
\end{align}
Using the Nitche technique we have
\begin{equation}
\fanshu{\sigma-\pror\sigma}\lesssim\eta_0(h)\fanshu[1]{\sigma-\pror\sigma}.\label{inequality2}
\end{equation}
Combining \eqref{inequality1} and \eqref{inequality2} we arrive at \eqref{L2sigErr} and \eqref{H1sigErr}.
\end{proof}

For any $(\psi,v)\in H_0^1(\Omega)\times H_0^1(\Omega),\ v\neq 0$, define the Rayleigh quotient
\begin{equation}
\lambda^R(v)=\frac{-a_h(\psi,\psi)+2b(\psi,v)}{(v,v)}.
\end{equation}
\begin{lemma}
Let $(\lambda,\sigma,u)$ be the eigenpair of \eqref{modifiedMixedFemEigen}, then for all $(\psi,v)\in H_0^1(\Omega)\times H_0^1(\Omega),\ v\neq0$, there holds
\begin{align}
\lambda^R(v)-\lambda&=\frac{-a(\psi-\sigma,\psi-\sigma)+2b(\psi-\sigma,v-u)}{(v,v)}-\frac{\lambda(v-u,v-u)}{(v,v)}\notag\\
&-\frac{(c(\psi)\xi-c_h(\psi)\xi_h,\psi)}{(v,v)}.\label{diffRay}
\end{align}
\end{lemma}
\begin{proof}
Note that
\begin{equation}
a_h(w,\psi)=a(w,\psi)+(c(w)\xi-c_h(w)\xi_h,\psi),\qquad\forall w,\psi\in H_0^1(\Omega).
\end{equation}
It is easy to see that
\begin{align*}
a(\psi,\psi)&=a(\psi-\sigma,\psi-\sigma)+2a(\psi,\sigma)-a(\sigma,\sigma),\\
b(\psi,v)&=b(\psi-\sigma,v-u)+b(\psi,u)+b(\sigma,v)-b(\sigma,u),\\
(v,v)&=(v-u,v-u)+2(v,u)-(u,u).
\end{align*}
Some tedious manipulation yields
\begin{align*}
-a_h(\psi,\psi)+2b(\psi,v)-\lambda(v,v)&=-a(\psi-\sigma,\psi-\sigma)+2b(\psi-\sigma,v-u)-\lambda(v-u,v-u)\\
&-(c(\psi)\xi-c_h(\psi)\xi_h,\psi)+(2b(\psi,u)-2a(\psi,\sigma))\\
&+(a(\sigma,\sigma)-2b(\sigma,u)+\lambda(u,u))+(2b(\sigma,v)-\lambda(v,u)).
\end{align*}
From \eqref{modifiedEigen} we can eliminate the last three terms. Dividing both sides by $(v,v)$, we have the desired conclusion.
\end{proof}
\begin{theorem}\label{TheoremIdentity}
Under the conditions of Theorem \ref{TheoremErrAll}, the following identity holds
\begin{equation}
\lambda_h-\lambda=(-a(\sigma_h-\sigma,\sigma_h-\sigma)+2b(\sigma_h-\sigma,u_h-u)-\lambda\fanshu{u_h-u}^2-(c(\sigma_h)\xi-c_h(\sigma_h)\xi_h,\sigma_h)).
\end{equation}
\end{theorem}
\begin{proof}
It follows from \eqref{modifiedMixedFemEigen} that
\begin{equation}
\lambda_h=\frac{-a_h(\sigma_h,\sigma_h)+2b(\sigma_h,u_h)}{(u_h,u_h)}.
\end{equation}
Taking $(\psi,v)=(\sigma_h,u_h)$ in \eqref{diffRay} we obtain the desired identity.
\end{proof}

From Theorem \ref{TheoremErrAll}, Theorem \ref{TheoremIdentity}, \eqref{eta0Estimate}, the interpolation theory and Lemma 3.3 in \cite{li2020biharmonic}, we can deduce the following theorem.
\begin{theorem}\label{theoremFine}
Under the conditions of Theorem \ref{TheoremErrAll}, the following inequalities hold
\begin{align}
\fanshu[1]{u-u_h}&\lesssim h,\\
\fanshu{u-u_h}&\lesssim h^{2\alpha},\\
\fanshu{\sigma-\sigma_h}+h^{\alpha}\fanshu[1]{\sigma-\sigma_h}&\lesssim h^{2\alpha},\\
|\lambda-\lambda_h|&\lesssim h^{2\alpha}.
\end{align}
\end{theorem}
\begin{remark}
If we use the graded meshes defined in \cite{li2020biharmonic}, then applying Lemma 4.6 in \cite{li2020biharmonic} and Theorem \ref{TheoremErrAll}, we can improve the result of Theorem \ref{theoremFine} as follows
\begin{align*}
\fanshu{u-u_h}+h\fanshu[1]{u-u_h}&\lesssim {h^{2}},\\
\fanshu{\sigma-\sigma_h}+h\fanshu[1]{\sigma-\sigma_h}&\lesssim h^{2},\\
|\lambda-\lambda_h|&\lesssim h^{2}.
\end{align*}
However, numerical experiments show that even if the uniform meshes are used in non-convex domains, the numerical eigenvalues can converge at $\mathcal{O}(h^2)$.
\end{remark}

\section{Numerical experiments}\label{SectionNumericalExp}
In this section, we present some numerical experiments to confirm the \textit{a priori} error estimates derived in Theorem \ref{theoremFine}. Since the usual mixed finite element method \eqref{usualformEigen} is valid in convex domain (see \cite{brenner2005c0interior,yang2018theadaptive}), here we consider only non-convex domains. We shall consider the following three different domains: the L-shaped domain, the slit domain, and the square ring, which are plotted in Figure \ref{figureDomains}. To implement the modified mixed finite element method \eqref{modifiedMixedFemEigen}, we need the following cut-off function\cite{li2020biharmonic}:
\begin{align*}
\chi(r;\tau,R)=
\begin{cases}
0,&\text{ if } r\ge R,\\
1,&\text{ if } r\le \tau R,\\
\frac{1}{2}\!-\!\frac{15}{16}\left(\frac{2r}{R(1\!-\!\tau)}\!-\!\frac{1+\tau}{1\!-\!\tau}\right)+\frac{5}{8}\left(\frac{2r}{R(1\!-\!\tau)}\!-\!\frac{1+\tau}{1\!-\!\tau}\right)^3\!-\!\frac{3}{16}\left(\frac{2r}{R(1\!-\!\tau)}\!-\!\frac{1+\tau}{1\!-\!\tau}\right)^5,&\text{ otherwise},
\end{cases}
\end{align*}
where $R=\frac{1}{4},\ \tau=\frac{1}{8}$ for the first two domains, and $R=\frac{1}{6},\ \tau=\frac{1}{8}$ for the last domain. All algorithms are implemented by MATLAB. The linear systems obtained by the modified mixed method are solved by the multigrid solver in the iFEM package \cite{chen2008ifem}.
\subsection{L-shaped domain}
The first experiment is for the L-shaped domain with $\Omega=(0,1)^2\setminus([1/2,1)\times(0,1/2])$. To begin with, we compare our method with the quadratic $C^0$ IPG, the quintic Argyris element, the Ciarlet-Raviart mixed method, the Morley element reported in \cite{brenner2015c0interior} and the quadratic and cubic Ciarlet-Raviart mixed method reported in \cite{yang2018theadaptive}. Our numerical method is computed on a quasi-uniform mesh generated by matlab code \verb"initmesh" with maximum edge size $=1/80$. The numerical results of the above methods are presented in Table \ref{TableComparison}. It can be seen that our method is comparable with $C^0$ IPG, Argyris and Morley methods, and has much less degrees of freedom than these methods. Furthermore, compared with the usual mixed method, C-R(2) and C-R(3), our method does not generate spurious eigenvalues in non-convex domain.

Next we investigate the convergence rate of the modified mixed method on uniform meshes. The initial mesh $\mathcal{T}_{h_0}$ is shown in Figure \ref{initialMesh}. The finer mesh $\mathcal{T}_{h_{l}}$ is obtained by dividing the initial mesh uniformly after $l$ times. Since the exact eigenvalues are unknown, we use the following numerical convergence rate
\begin{align}
\mathcal{R} = \log_2\frac{|\lambda_{i,h_l}-\lambda_{i,h_{l-1}}|}{|\lambda_{i,h_{l+1}}-\lambda_{i,h_{l}}|}\label{rateIndicator}
\end{align}
as an indicator of the actual convergence rate. Here $\lambda_{i,h_l}$ denotes the $i$-th numerical eigenvalue obtained by the modified mixed method on the mesh $\mathcal{T}_{h_l}$. Table \ref{EigensTest1} lists the first six numerical eigenvalues on the L-shaped domain using uniform meshes. Table \ref{rateLshaped} records the convergence history of the first six eigenvalues obtained by the modified mixed method. Surprisingly, despite the re-entrant corner, the convergence rates are all $\mathcal{O}(h^2)$, which is better than the rate $\mathcal{O}(h^{2\alpha})$ predicted in Theorem \ref{theoremFine}.
\begin{figure}[!htbp]\centering
\subfigure[\label{initialMesh}]{\includegraphics[scale=0.5,viewport=82 29 351 297,clip=true]{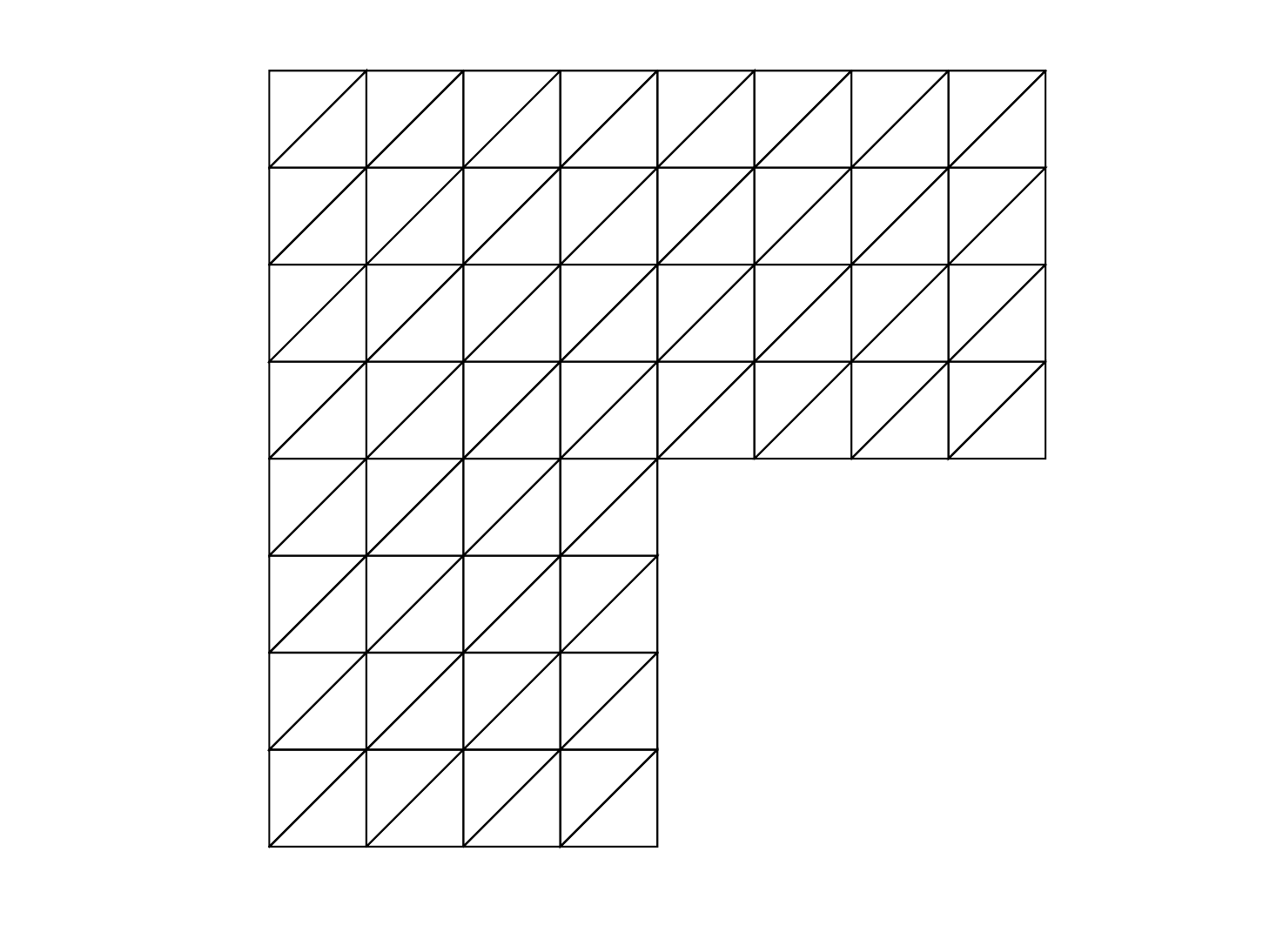}}
\subfigure[\label{initialMeshSlit}]{\includegraphics[scale=0.5,viewport=82 29 351 297,clip=true]{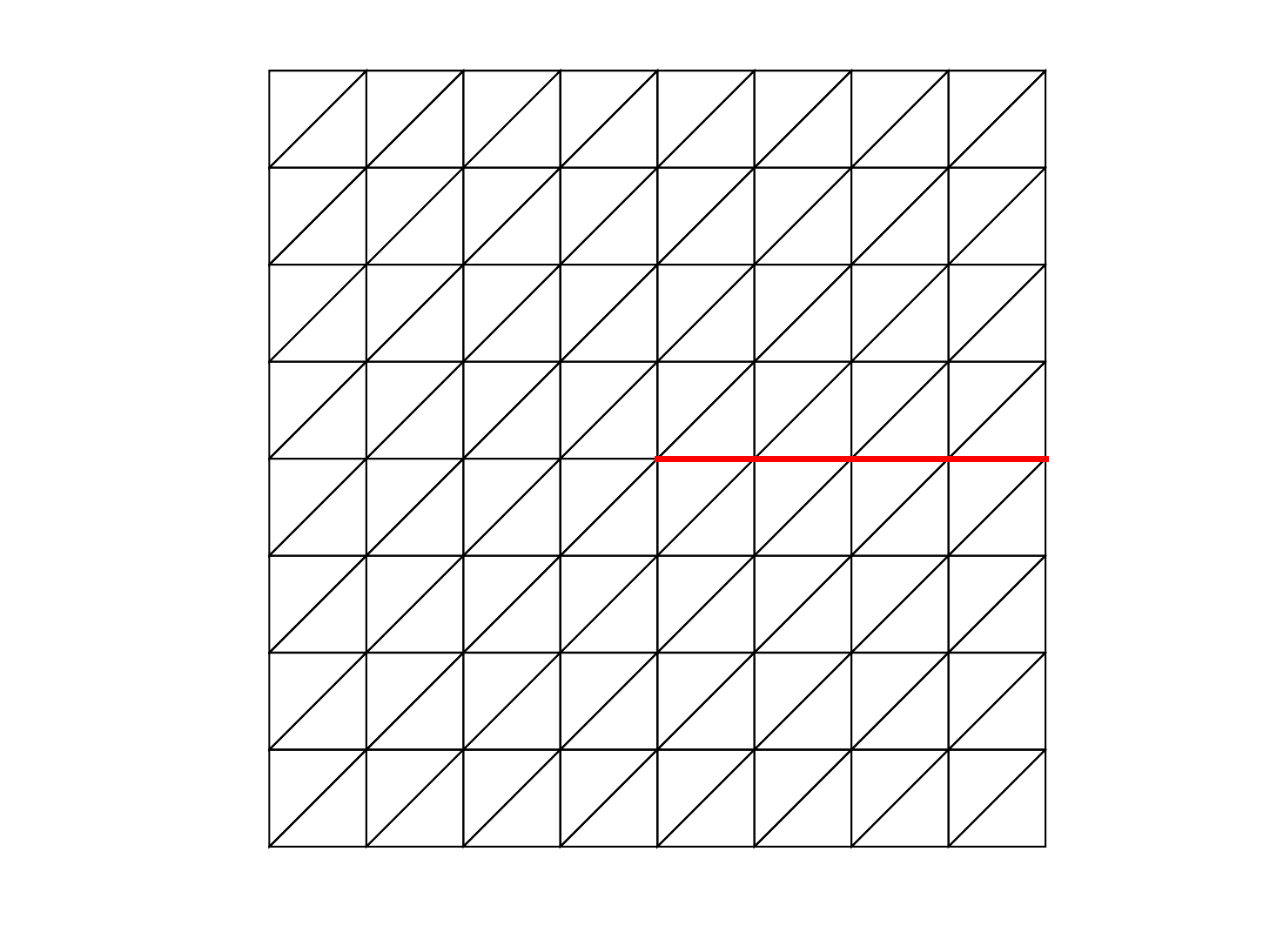}}
\subfigure[\label{initialMeshHole}]{\includegraphics[scale=0.5,viewport=82 29 351 297,clip=true]{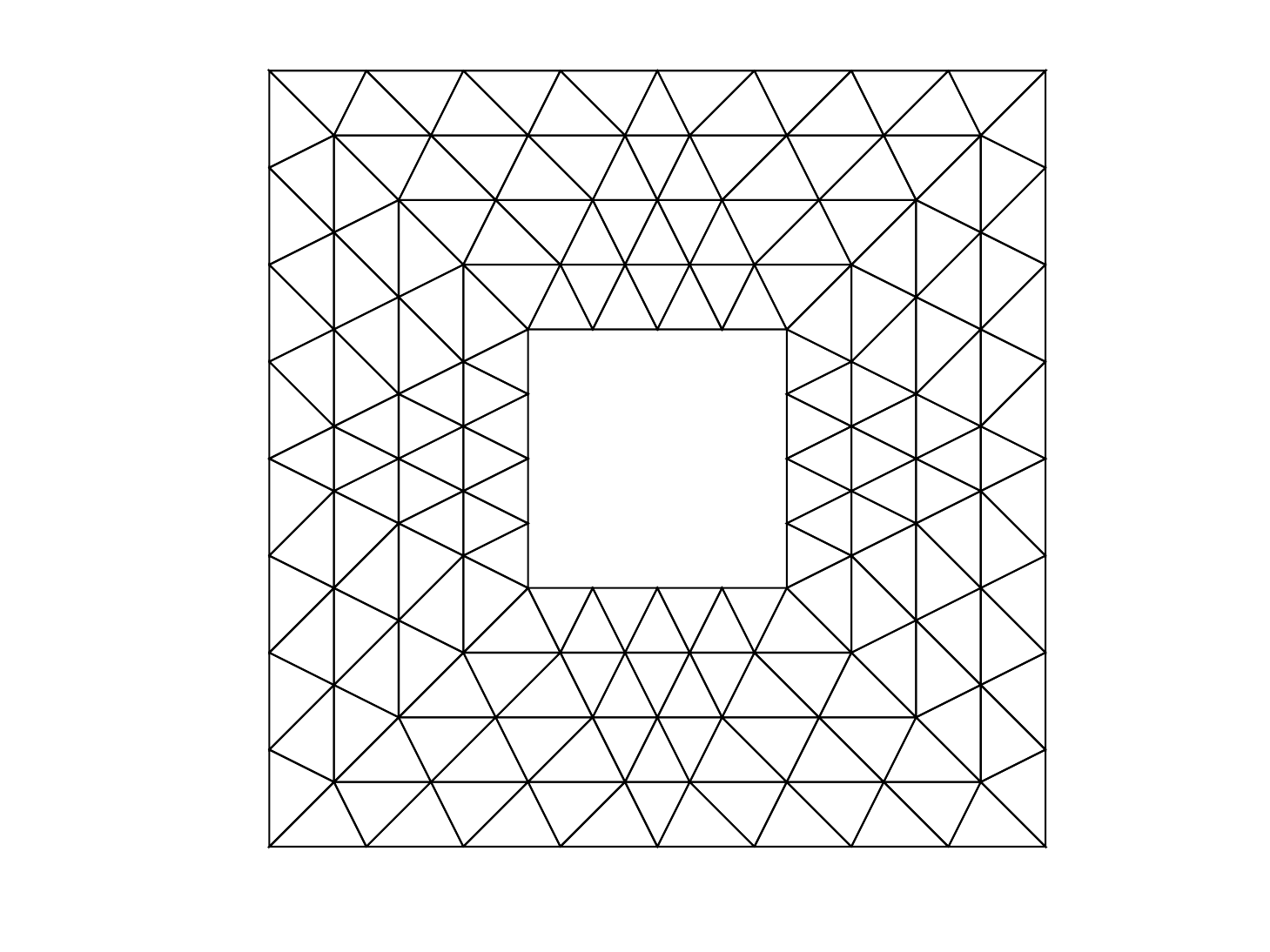}}
\caption{Initial meshes for (a) the L-shaped domain, (b) the slit domain and (c) the square ring.\label{figureDomains}}
\end{figure}
\begin{table}[!htbp]\centering
\caption{The first six eigenvalues obtained by seven methods on the L-shaped domain.\label{TableComparison}}
\scalebox{0.9}[0.9]{
\begin{tabular}{*{8}{c@{\extracolsep{5pt}}@{\extracolsep{5pt}}}}\hline
    &dof    &$\lambda_{1,h}$ & $\lambda_{2,h}$ & $\lambda_{3,h}$ &$\lambda_{4,h}$  &$\lambda_{5,h}$& $\lambda_{6,h}$\\\hline
 Modified Mixed                             &8077 & 2621& 3699 & 6242& 13968& 19234& 31021 \\
 $C^0$IPG\cite{brenner2015c0interior}       &32705 & 2718& 3743 & 6061& 13666& 19156& 31027 \\
 Argyris\cite{brenner2015c0interior}        &74454 & 2692& 3765 & 6234& 13972& 19375& 31281 \\
 Morley \cite{brenner2015c0interior}        &33025 & 2414& 3663 & 6225& 13904& 18642& 30002 \\
 Mixed \cite{brenner2015c0interior}         &8097 & 1491& 3699 & 6242& 13969& 16354& 27617 \\
 C-R(2) \cite{yang2018theadaptive}          &5890 & 1490& 3695 & 6234& 13945& 16319& 27547 \\
 C-R(3) \cite{yang2018theadaptive}          &3266 & 1490& 3695 & 6234& 13944& 16318& 27545 \\
\hline
\end{tabular}}
\end{table}
\begin{table}[!htbp]\centering
\caption{The first six eigenvalues obtained by modified mixed method on the L-shaped domain using uniform meshes.\label{EigensTest1}}
\scalebox{0.9}[0.9]{
\begin{tabular}{*{7}{c@{\extracolsep{5pt}}@{\extracolsep{5pt}}}}\hline
    dof   &$\lambda_{1,h}$ & $\lambda_{2,h}$ & $\lambda_{3,h}$ &$\lambda_{4,h}$  &$\lambda_{5,h}$& $\lambda_{6,h}$\\\hline
          47   &  2822.2404  &  4595.5512  &  8342.5404  &  20993.3799  &  27482.8942  &  48416.0562  \\
         191   &  2673.3623  &  3910.3929  &  6726.3337  &  15565.4487  &  21179.2559  &  34840.8046  \\
         767   &  2633.7149  &  3748.6548  &  6355.0676  &  14341.0616  &  19692.8282  &  31894.6539  \\
        3071   &  2623.4781  &  3708.6511  &  6264.2695  &  14042.9825  &  19322.7808  &  31183.8024  \\
       12287   &  2620.7658  &  3698.6468  &  6241.6955  &  13968.9335  &  19229.8576  &  31007.0613  \\
       49151   &  2620.0725  &  3696.1406  &  6236.0597  &  13950.4466  &  19206.5326  &  30962.7136  \\
      196607   &  2619.8900  &  3695.5130  &  6234.6513  &  13945.8259  &  19200.6648  &  30951.5609  \\
      786431   &  2619.8424  &  3695.3559  &  6234.2992  &  13944.6707  &  19199.1893  &  30948.7547  \\
     3145727   &  2619.8300  &  3695.3166  &  6234.2112  &  13944.3818  &  19198.8182  &  30948.0485  \\
    125829111  &  2619.8268  &  3695.3067  &  6234.1892  &  13944.3096  &  19198.7249  &  30947.8708  \\
\hline
\end{tabular}}
\end{table}
\begin{table}[!htbp]\centering
\caption{The convergence history of the first six eigenvalues in the L-shaped domain.\label{rateLshaped}}
\scalebox{0.7}[0.7]{
\begin{tabular}{*{13}{c@{\extracolsep{5pt}}@{\extracolsep{5pt}}}}\hline
  $h$ &$|\lambda_{1,h_{l+1}}-\lambda_{1,h_l}|$&  $\mathcal{R}$ & $|\lambda_{2,h_{l+1}}-\lambda_{2,h_l}|$&  $\mathcal{R}$  & $|\lambda_{3,h_{l+1}}-\lambda_{3,h_l}|$ &  $\mathcal{R}$ &$|\lambda_{4,h_{l+1}}-\lambda_{4,h_l}|$ &  $\mathcal{R}$ &$|\lambda_{5,h_{l+1}}-\lambda_{5,h_l}|$ &  $\mathcal{R}$ & $|\lambda_{6,h_{l+1}}-\lambda_{6,h_l}|$ &  $\mathcal{R}$\\\hline
0.1768& 1.489e+02&      & 6.852e+02&      & 1.616e+03&      & 5.428e+03&      & 6.304e+03&      & 1.358e+04&   \\
0.0884& 3.965e+01& 1.91& 1.617e+02& 2.08& 3.713e+02& 2.12& 1.224e+03& 2.15& 1.486e+03& 2.08& 2.946e+03& 2.20   \\
0.0442& 1.024e+01& 1.95& 4.000e+01& 2.02& 9.080e+01& 2.03& 2.981e+02& 2.04& 3.700e+02& 2.01& 7.109e+02& 2.05   \\
0.0221& 2.712e+00& 1.92& 1.000e+01& 2.00& 2.257e+01& 2.01& 7.405e+01& 2.01& 9.292e+01& 1.99& 1.767e+02& 2.01   \\
0.0110& 6.932e--01& 1.97& 2.506e+00& 2.00& 5.636e+00& 2.00& 1.849e+01& 2.00& 2.333e+01& 1.99& 4.435e+01& 1.99   \\
0.0055& 1.825e--01& 1.93& 6.277e--01& 2.00& 1.408e+00& 2.00& 4.621e+00& 2.00& 5.868e+00& 1.99& 1.115e+01& 1.99   \\
0.0028& 4.761e--02& 1.94& 1.571e--01& 2.00& 3.521e--01& 2.00& 1.155e+00& 2.00& 1.476e+00& 1.99& 2.806e+00& 1.99   \\
0.0014& 1.239e--02& 1.94& 3.931e--02& 2.00& 8.802e--02& 2.00& 2.888e--01& 2.00& 3.711e--01& 1.99& 7.062e--01& 1.99   \\
0.0007& 3.211e--03& 1.95& 9.832e--03& 2.00& 2.200e--02& 2.00& 7.221e--02& 2.00& 9.330e--02& 1.99& 1.777e--01& 1.99   \\
\hline
\end{tabular}}
\end{table}

\subsection{Slit domain}\label{numericaltest2}
In the second experiment, we consider the slit domain $\Omega=(0,1)^2\setminus\{1/2\le x\le 1,\ y=1/2\}$, with maximal re-entrant corner of angle $2\pi$. The initial mesh on this domain is shown in Figure \ref{initialMeshSlit}. We also refine the initial mesh uniformly to get finer meshes. Since the exact eigenvalues are unknown, we also use \eqref{rateIndicator} as an indicator of the actual convergence rate. In Table \ref{EigensTest2}, we present the first six numerical eigenvalues obtained by the modified mixed method on the slit domain. In Table \ref{rateslit}, we display the convergence history of these eigenvalues. As the previous experiment, although the eigenfunctions are of low regularity around the corner, all numerical eigenvalues on the uniform meshes achieve the optimal rate $\mathcal{O}(h^2)$. In Table \ref{TableComparisonSlit}, we compare the modified mixed method with other four methods. The mesh size used in the computations is $\approx 0.0055.$ Comparing with the modified mixed method, the $C^0$ IPG method, the Morley element and the Argyris element behave better when the eigenfunctions are very smooth. However, the performance of the modified mixed method is better than other methods when the eigenfunctions are less smooth. Like the previous experiment, we observe from Table \ref{TableComparisonSlit} that the usual mixed finite element method may produce spurious eigenvalues when the eigenfunctions are less smooth, while the modified mixed method is free of spurious eigenvalues.

\begin{table}\centering
\caption{The first six eigenvalues obtained by modified mixed method on the slit domain.\label{EigensTest2}}
\scalebox{0.9}[0.9]{
\begin{tabular}{*{7}{c@{\extracolsep{5pt}}@{\extracolsep{5pt}}}}\hline
    dof   &$\lambda_{1,h}$ & $\lambda_{2,h}$ & $\lambda_{3,h}$ &$\lambda_{4,h}$  &$\lambda_{5,h}$& $\lambda_{6,h}$\\\hline
        45  &  2765.8409  &  2888.8669  &  5544.7613  &  8338.4977  &  17802.8847  &  24871.4901  \\
      217  &  2539.5982  &  2706.2879  &  4699.0746  &  6725.9941  &  13760.5363  &  18466.1474  \\
      945  &  2461.1212  &  2689.6879  &  4498.9558  &  6355.0434  &  12827.9667  &  16955.2271  \\
      3937  &  2441.6869  &  2686.0276  &  4449.5111  &  6264.2678  &  12599.5952  &  16584.8771  \\
     16065  &  2436.8412  &  2685.1205  &  4437.1681  &  6241.6953  &  12542.7854  &  16492.7877  \\
     64897  &  2435.6307  &  2684.9069  &  4434.0811  &  6236.0597  &  12528.5991  &  16469.7971  \\
    260865  &  2435.3281  &  2684.8510  &  4433.3090  &  6234.6513  &  12525.0534  &  16464.0514  \\
   1046017  &  2435.2525  &  2684.8372  &  4433.1159  &  6234.2992  &  12524.1670  &  16462.6151  \\
   4189185  &  2435.2336  &  2684.8336  &  4433.0677  &  6234.2112  &  12523.9454  &  16462.2561  \\
  16766977  &  2435.2289  &  2684.8327  &  4433.0556  &  6234.1892  &  12523.8900  &  16462.1663  \\
\hline
\end{tabular}}
\end{table}

\begin{table}[!htbp]\centering
\caption{The convergence history of the first six eigenvalues in the slit domain.\label{rateslit}}
\scalebox{0.7}[0.7]{
\begin{tabular}{*{13}{c@{\extracolsep{5pt}}@{\extracolsep{5pt}}}}\hline
  $h$ &$|\lambda_{1,h_{l+1}}-\lambda_{1,h_l}|$&  $\mathcal{R}$ & $|\lambda_{2,h_{l+1}}-\lambda_{2,h_l}|$&  $\mathcal{R}$  & $|\lambda_{3,h_{l+1}}-\lambda_{3,h_l}|$ &  $\mathcal{R}$ &$|\lambda_{4,h_{l+1}}-\lambda_{4,h_l}|$ &  $\mathcal{R}$ &$|\lambda_{5,h_{l+1}}-\lambda_{5,h_l}|$ &  $\mathcal{R}$ & $|\lambda_{6,h_{l+1}}-\lambda_{6,h_l}|$ &  $\mathcal{R}$\\\hline
0.1768& 2.262e+02&      & 1.826e+02&      & 8.457e+02&      & 1.613e+03&      & 4.042e+03&      & 6.405e+03&   \\
0.0884& 7.848e+01& 1.53& 1.660e+01& 3.46& 2.001e+02& 2.08& 3.710e+02& 2.12& 9.326e+02& 2.12& 1.511e+03& 2.08   \\
0.0442& 1.943e+01& 2.01& 3.660e+00& 2.18& 4.944e+01& 2.02& 9.078e+01& 2.03& 2.284e+02& 2.03& 3.703e+02& 2.03   \\
0.0221& 4.846e+00& 2.00& 9.071e--01& 2.01& 1.234e+01& 2.00& 2.257e+01& 2.01& 5.681e+01& 2.01& 9.209e+01& 2.01   \\
0.0110& 1.211e+00& 2.00& 2.136e--01& 2.09& 3.087e+00& 2.00& 5.636e+00& 2.00& 1.419e+01& 2.00& 2.299e+01& 2.00   \\
0.0055& 3.026e--01& 2.00& 5.583e--02& 1.94& 7.721e--01& 2.00& 1.408e+00& 2.00& 3.546e+00& 2.00& 5.746e+00& 2.00   \\
0.0028& 7.564e--02& 2.00& 1.387e--02& 2.01& 1.931e--01& 2.00& 3.521e--01& 2.00& 8.864e--01& 2.00& 1.436e+00& 2.00   \\
0.0014& 1.891e--02& 2.00& 3.580e--03& 1.95& 4.828e--02& 2.00& 8.802e--02& 2.00& 2.216e--01& 2.00& 3.591e--01& 2.00   \\
0.0007& 4.727e--03& 2.00& 8.850e--04& 2.02& 1.207e--02& 2.00& 2.200e--02& 2.00& 5.540e--02& 2.00& 8.977e--02& 2.00   \\
\hline
\end{tabular}}
\end{table}

\begin{table}[!htbp]\centering
\caption{The first six eigenvalues obtained by five methods on the slit domain.\label{TableComparisonSlit}}
\scalebox{0.9}[0.9]{
\begin{tabular}{*{8}{c@{\extracolsep{5pt}}@{\extracolsep{5pt}}}}\hline
    &dof    &$\lambda_{1,h}$ & $\lambda_{2,h}$ & $\lambda_{3,h}$ &$\lambda_{4,h}$  &$\lambda_{5,h}$& $\lambda_{6,h}$\\\hline
 Modified Mixed &16065  &  2436.8412  &  2685.1205  &  4437.1681  &  6241.6953  &  12542.7854  &  16492.7877  \\
 $C^0$IPG       &260865  &  2436.5873  &  2718.3221  &  4453.8718  &  6242.5260  &  12545.4672  &  16492.9858  \\
 Argyris        &590458  &  2435.2273  &  2691.6942  &  4438.9371  &  6234.1818  &  12525.8683  &  16462.1364  \\
 Morley         &262145  &  2434.8270  &  2642.2727  &  4428.3125  &  6232.8523  &  12517.9292  &  16455.8866  \\
 Mixed          &16065  &  1133.0888  &  2436.8413  &  4437.1681  &  6241.6951  &  12542.7854  &  15053.5876  \\
\hline
\end{tabular}}
\end{table}

\subsection{Square ring}
Finally, we consider the square ring $\Omega = (0,1)^2\setminus([1/3,2/3]^2)$. In this experiment, we use quasi-uniform meshes. The initial mesh on this domain is shown in Figure \ref{initialMeshHole}, and we refine the initial mesh uniformly to obtain finer meshes. The numerical eigenvalues and related convergence history are reported in Tables \ref{EigensTest3} and \ref{ratehole}, respectively. It can be seen from Table \ref{ratehole} that despite the four re-entrant corners all numerical eigenvalues converge at the optimal rate $\mathcal{O}(h^2)$. Note that the convergence rate is better than what we obtained in Theorem \ref{theoremFine}. In Table \ref{TableComparisonHole}, we display numerical results of five different methods on a quasi-uniform mesh with mesh size $\approx0.0039.$ It can be seen that the performance of the modified mixed method is comparable to the $C^0$IPG method, the Argyris method and the Morley method. In addtion, comparing with the usual mixed method, our method does not generate spurious eigenvalues.

It worth mentioning that our method has certain advantages in approximating multiple eigenvalues. In this experiment, the second and the third eigenvalues are repeated, i.e. $\lambda_2=\lambda_3.$ The corresponding eigenfunctions are depicted in Figure \ref{eigenfunshole}. It can be seen that the second eigenfunction has singularities at the upper-left and low-right of the hole, while the third eigenfunction is singular at the other two opposite corners. To capture the singularities, a popular way is to use adaptive scheme. Several studies have shown that an effective adaptive strategy for multiple eigenvalues should consider all involved discrete eigenfunctions\cite{boffi2017posteriori,solin2012eigen,dai2008convergence}, otherwise the singularities of the target eigenfunctions may be resolved in a wrong way. However, the main issue is that in general it is not known \textit{a priori} the multiplicity of an eigenvalue of the continuous problem, which in turn brings difficulties to the design of efficient adaptive methods.

Interestingly, it seems that our method does not need to consider these problems. As mentioned before, our method enjoys $\mathcal{O}(h^2)$ convergence rate for the first six eigenvalues on quasi-uniform meshes, which means that there is no need to use adaptive method. Therefore, we do not need to bother to design efficient adaptive strategies, especially when the multiplicity of an eigenvalue is not known \textit{a priori}. Moreover, due to the symmetry of the initial mesh and the uniform refinement strategy, numerical eigenvalues which approximate a multiple eigenvalue can be very close. For instance $\lambda_{2,h}$ and $\lambda_{3,h}$ in Table \ref{EigensTest3} are almost equal as $h\rightarrow 0$. To sum up, our method has some potential in calculating multiple eigenvalues.

\begin{table}[!htbp]\centering
\caption{The first six eigenvalues obtained by modified mixed method on the square ring.\label{EigensTest3}}
\scalebox{0.9}[0.9]{
\begin{tabular}{*{7}{c@{\extracolsep{5pt}}@{\extracolsep{5pt}}}}\hline
    dof   &$\lambda_{1,h}$ & $\lambda_{2,h}$ & $\lambda_{3,h}$ &$\lambda_{4,h}$  &$\lambda_{5,h}$& $\lambda_{6,h}$\\\hline
    72  &  13199.3873  &  15043.5260  &  15221.0294  &  16069.9131  &  19530.5875  &  26203.3113  \\
    336  &  11833.8354  &  12463.7720  &  12536.8434  &  14675.5949  &  16125.1736  &  22760.3983  \\
   1440  &  11618.2080  &  12247.8454  &  12248.3350  &  14321.0590  &  15722.7245  &  21997.2147  \\
   5952  &  11586.6967  &  12204.9675  &  12204.9747  &  14231.2062  &  15644.8961  &  21808.1987  \\
  24192  &  11578.4400  &  12193.8534  &  12193.8536  &  14208.5011  &  15625.2749  &  21760.9347  \\
  97536  &  11576.3175  &  12191.0105  &  12191.0105  &  14202.7829  &  15620.2472  &  21749.0480  \\
 391680  &  11575.7742  &  12190.2890  &  12190.2890  &  14201.3466  &  15618.9871  &  21746.0760  \\
1569792  &  11575.6342  &  12190.1048  &  12190.1048  &  14200.9864  &  15618.6663  &  21745.3307  \\
6285312  &  11575.5987  &  12190.0583  &  12190.0583  &  14200.8962  &  15618.5853  &  21745.1440  \\
\hline
\end{tabular}}
\end{table}

\begin{table}[!htbp]\centering
\caption{The convergence history of the first six eigenvalues in the square ring.\label{ratehole}}
\scalebox{0.7}[0.7]{
\begin{tabular}{*{13}{c@{\extracolsep{5pt}}@{\extracolsep{5pt}}}}\hline
  $h$ &$|\lambda_{1,h_{l+1}}-\lambda_{1,h_l}|$&  $\mathcal{R}$ & $|\lambda_{2,h_{l+1}}-\lambda_{2,h_l}|$&  $\mathcal{R}$  & $|\lambda_{3,h_{l+1}}-\lambda_{3,h_l}|$ &  $\mathcal{R}$ &$|\lambda_{4,h_{l+1}}-\lambda_{4,h_l}|$ &  $\mathcal{R}$ &$|\lambda_{5,h_{l+1}}-\lambda_{5,h_l}|$ &  $\mathcal{R}$ & $|\lambda_{6,h_{l+1}}-\lambda_{6,h_l}|$ &  $\mathcal{R}$\\\hline
0.1250& 1.366e+03&      & 2.580e+03&      & 2.684e+03&      & 1.394e+03&      & 3.405e+03&      & 3.443e+03&   \\
0.0625& 2.156e+02& 2.66& 2.159e+02& 3.58& 2.885e+02& 3.22& 3.545e+02& 1.98& 4.024e+02& 3.08& 7.632e+02& 2.17   \\
0.0313& 3.151e+01& 2.77& 4.288e+01& 2.33& 4.336e+01& 2.73& 8.985e+01& 1.98& 7.783e+01& 2.37& 1.890e+02& 2.01   \\
0.0156& 8.257e+00& 1.93& 1.111e+01& 1.95& 1.112e+01& 1.96& 2.271e+01& 1.98& 1.962e+01& 1.99& 4.726e+01& 2.00   \\
0.0078& 2.122e+00& 1.96& 2.843e+00& 1.97& 2.843e+00& 1.97& 5.718e+00& 1.99& 5.028e+00& 1.96& 1.189e+01& 1.99   \\
0.0039& 5.433e-01& 1.97& 7.215e-01& 1.98& 7.215e-01& 1.98& 1.436e+00& 1.99& 1.260e+00& 2.00& 2.972e+00& 2.00   \\
0.0020& 1.400e-01& 1.96& 1.842e-01& 1.97& 1.842e-01& 1.97& 3.602e-01& 2.00& 3.207e-01& 1.97& 7.453e-01& 2.00   \\
0.0010& 3.546e-02& 1.98& 4.650e-02& 1.99& 4.650e-02& 1.99& 9.022e-02& 2.00& 8.102e-02& 1.98& 1.867e-01& 2.00   \\
\hline
\end{tabular}}
\end{table}

\begin{table}[!htbp]\centering
\caption{The first six eigenvalues obtained by five methods on the square ring.\label{TableComparisonHole}}
\scalebox{0.9}[0.9]{
\begin{tabular}{*{8}{c@{\extracolsep{5pt}}@{\extracolsep{5pt}}}}\hline
    &dof    &$\lambda_{1,h}$ & $\lambda_{2,h}$ & $\lambda_{3,h}$ &$\lambda_{4,h}$  &$\lambda_{5,h}$& $\lambda_{6,h}$\\\hline
 Modified Mixed &97536  &  11576.318  &  12191.010  &  12191.010  &  14202.783  &  15620.247  &  21749.048  \\
 $C^0$IPG       &391680  &  11936.788  &  12541.947  &  12541.947  &  14484.496  &  15993.189  &  21995.884  \\
 Argyris        &885488  &  11771.381  &  12391.618  &  12391.618  &  14411.874  &  15819.480  &  21914.541  \\
 Morley         &393216  &  10950.898  &  11651.220  &  11651.220  &  14105.090  &  14971.428  &  21571.899  \\
 Mixed          & 97536  &   6008.641  &   7149.365  &   7149.365  &   9659.918  &  14202.783  &  21052.856  \\
\hline
\end{tabular}}
\end{table}

\begin{figure}[!htp]
\makeatletter
\renewcommand{\@thesubfigure}{\hskip\subfiglabelskip}
 \makeatother
\centering
\subfigure{\includegraphics[scale=0.35,viewport=81 189 525 590,clip=true]{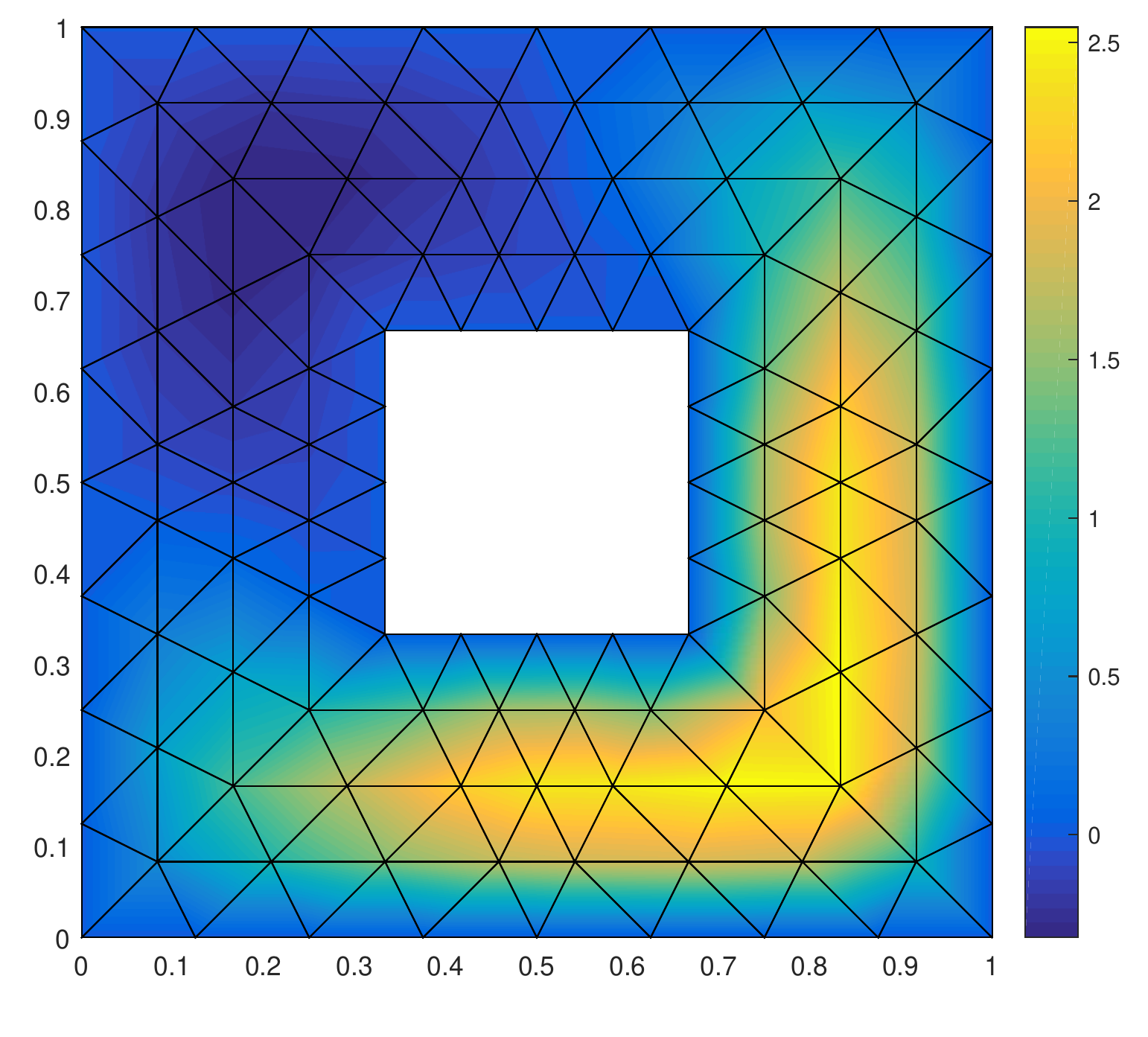}}\hspace{.2in}
\subfigure{\includegraphics[scale=0.35,viewport=81 189 525 590,clip=true]{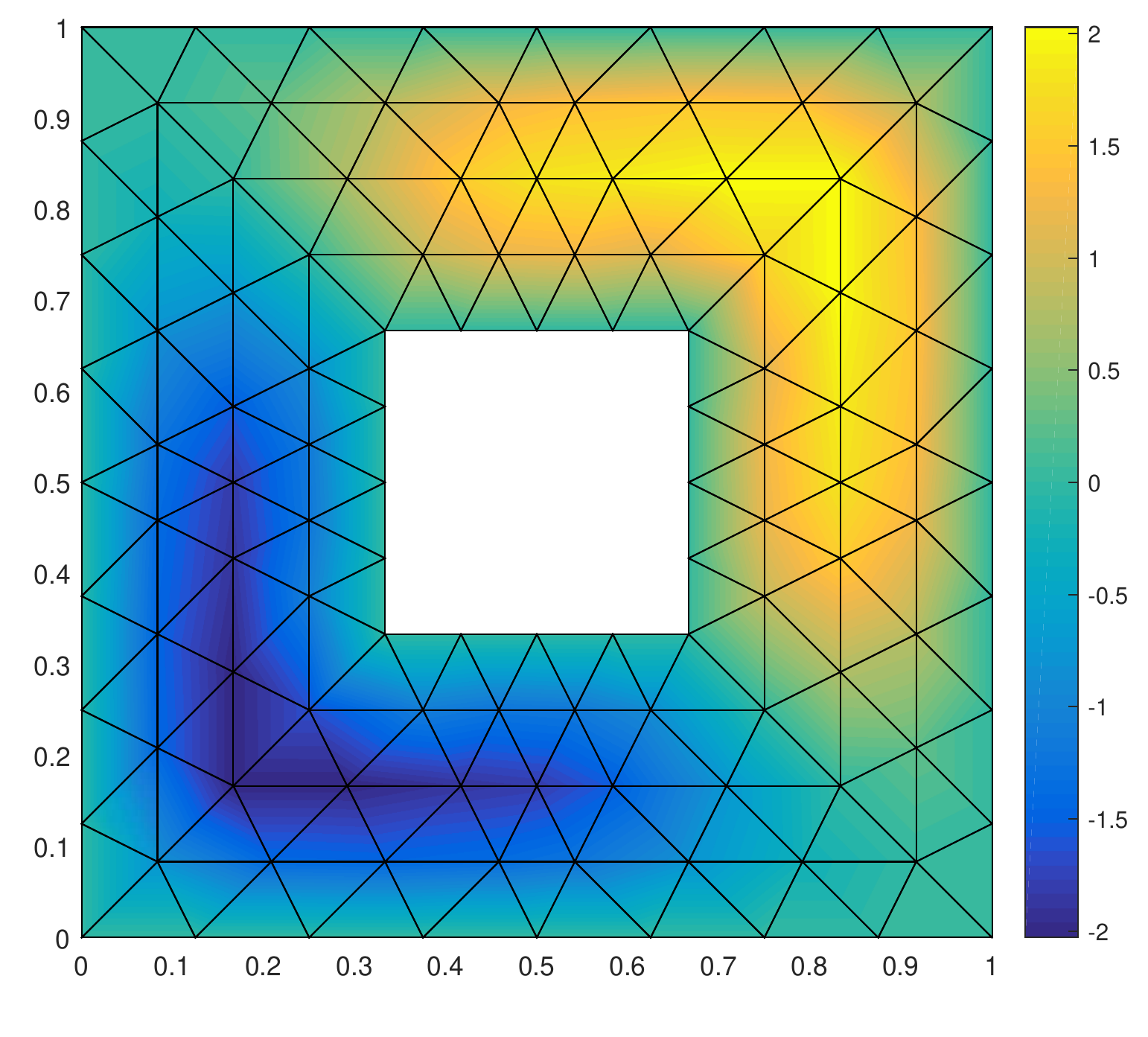}}\\
\subfigure{\includegraphics[scale=0.35,viewport=81 189 525 590,clip=true]{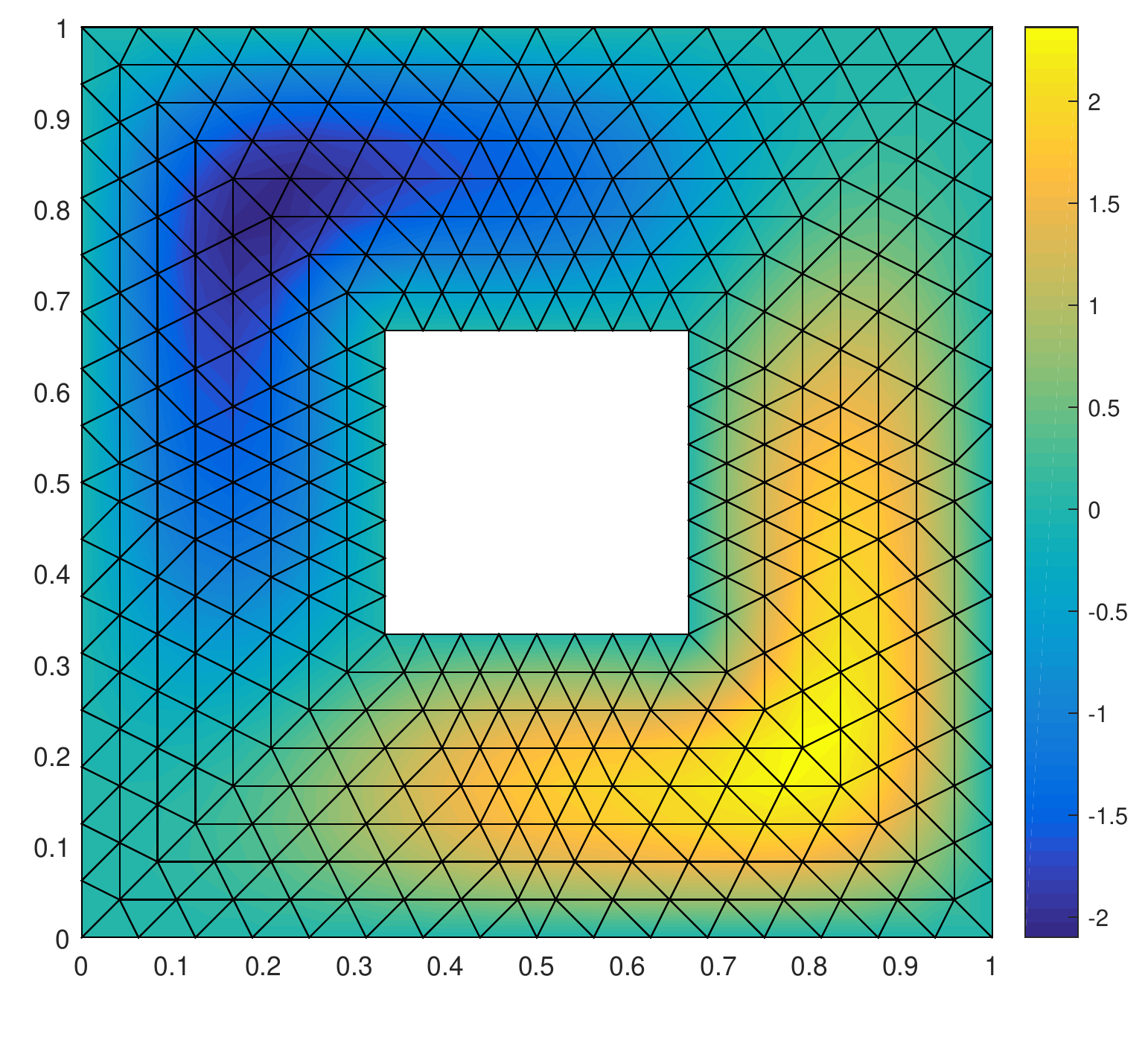}}\hspace{.2in}
\subfigure{\includegraphics[scale=0.35,viewport=81 189 525 590,clip=true]{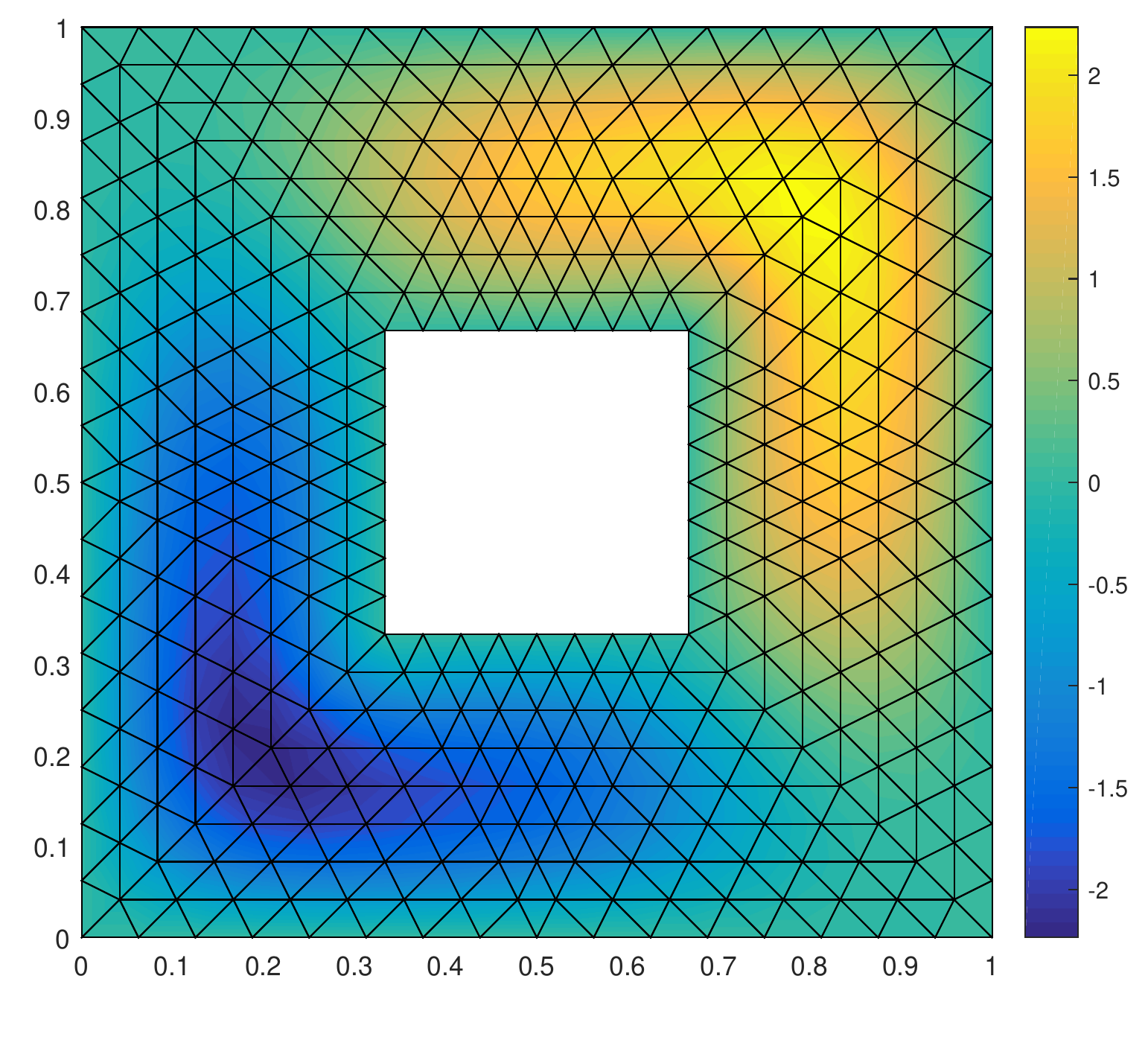}}\\
\subfigure{\includegraphics[scale=0.35,viewport=81 189 525 590,clip=true]{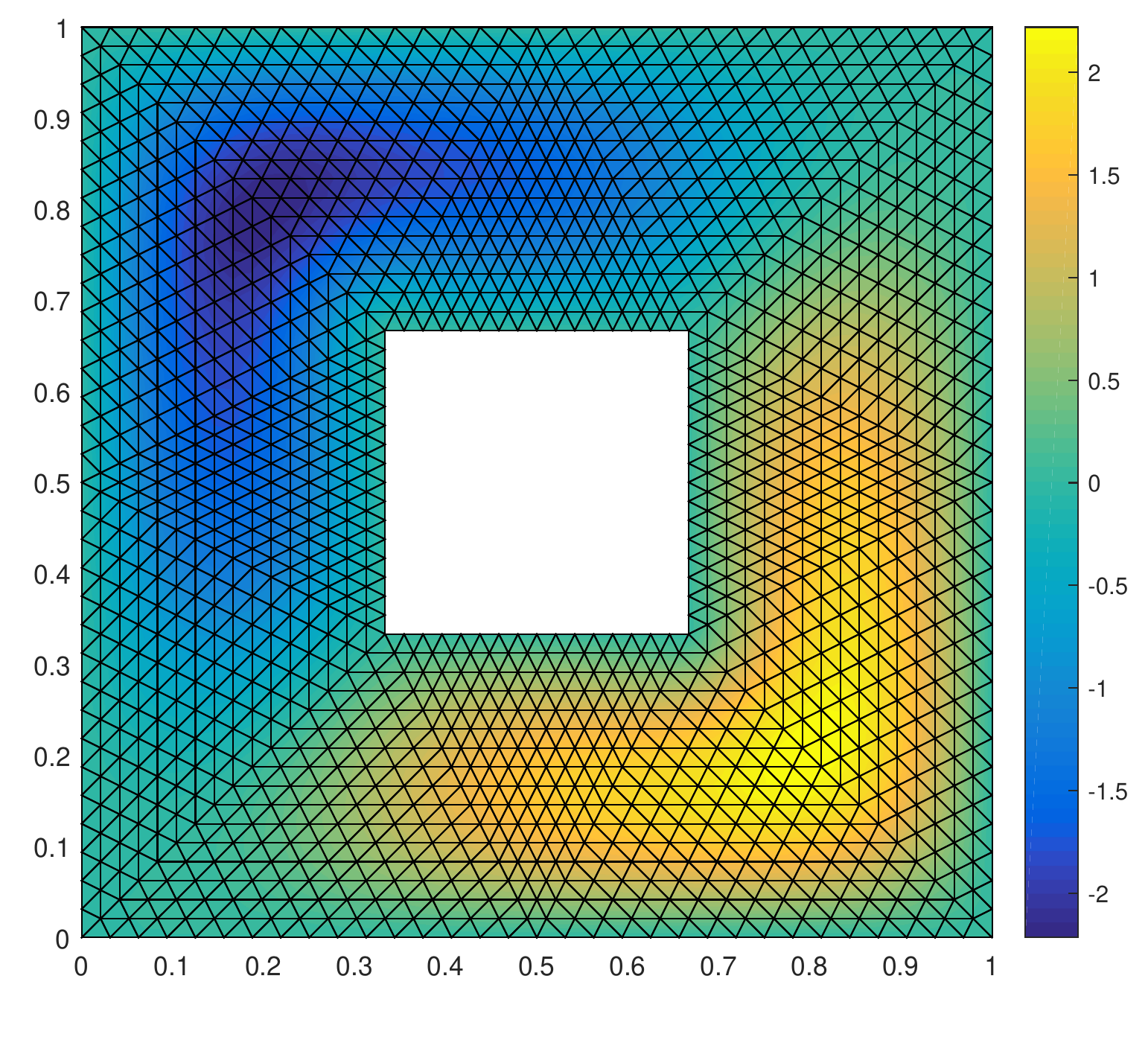}}\hspace{.2in}
\subfigure{\includegraphics[scale=0.35,viewport=81 189 525 590,clip=true]{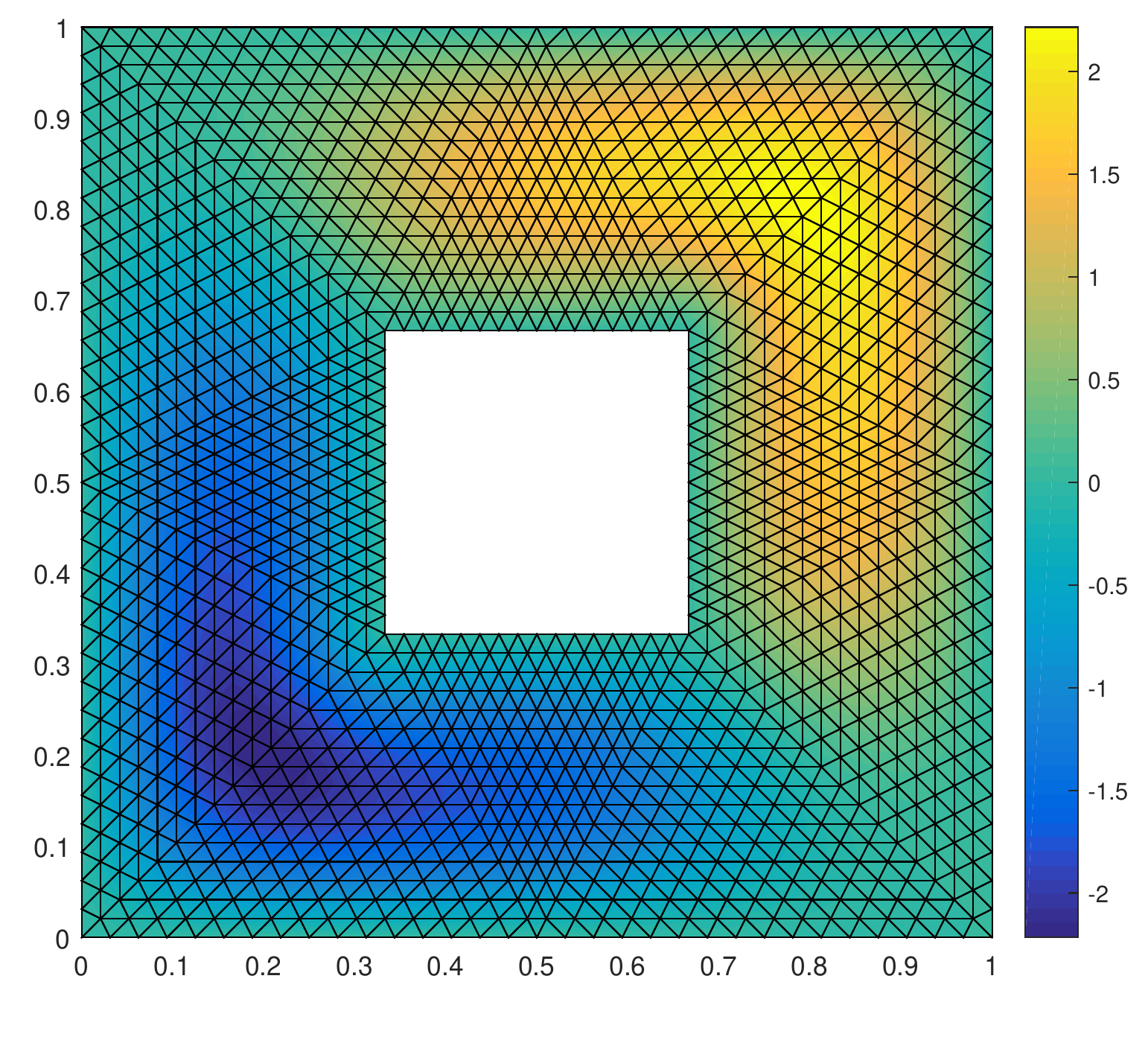}}\\
\caption{The eigenfunctions corresponding to $\lambda_2$ (left) and $\lambda_3$ (right) on the square ring for the three refinement levels \label{eigenfunshole}}
\end{figure}

\section*{Acknowledgments}
The authors would like to thanks anonymous referees for their valuable comments. This work is supported in part by the National Natural Science Foundation of China grants NSFC 11871092 and NSAF U1930402.
\bibliographystyle{abbrv}
\bibliography{MixedFemBiEigen}
\end{document}